\documentclass[a4paper,11pt,reqno]{amsart}
\usepackage[utf8]{inputenc}

\usepackage{amsmath,amsthm,amssymb}
\usepackage{a4wide}
\usepackage{txfonts,enumerate,pdfsync}
\usepackage{mathtools}
\usepackage[numbers]{natbib}
\usepackage{hyphenat}
\usepackage{hyperref}

\allowdisplaybreaks[1]

\usepackage{microtype}
\makeatletter
\def\MT@register@subst@font{\MT@exp@one@n\MT@in@clist\font@name\MT@font@list
   \ifMT@inlist@\else\xdef\MT@font@list{\MT@font@list\font@name,}\fi}
\makeatother 

\newcommand{\X}{\mathbf{X}}

\newcommand{\ZZ}{\mathbf{Z}}
\newcommand{\1}{\mathbf{1}}
\newcommand{\Z}{\mathbb{Z}}
\newcommand{\R}{\mathbb{R}}
\newcommand{\N}{\mathbb{N}}
\newcommand{\E}{\mathbb{E}}
\newcommand{\wt}{\widetilde}
\newcommand{\wh}{\widehat}
\newcommand{\wtt}{\overline}

\newcommand{\C}{\mathbb{C}}
\newcommand{\oc}{\overline c}

\newcommand{\tr}{\operatorname{tr}}
\newcommand{\var}{\operatorname{\mathbb{V}ar}}

\newcommand{\dd}{\mathrm{d}}
\newcommand{\ee}{\mathrm{e}}
\newcommand{\ii}{\mathrm{i}}

\newcommand{\T}{\mathsf{T}}

\newcommand{\matchA}{\overbracket[.5pt][5pt]{\:\:}{\,\!}\:\overbracket[.5pt][5pt]{\:\:}{\,\!}}
\newcommand{\matchB}{\overbracket[.5pt][5pt]{\:\:\:\:}{\,\!}\!\!\!\!\!\!\overbracket[.5pt]{\:\:\:\:}{\,\!}}
\newcommand{\matchC}{\overbracket[.5pt][5pt]{\:\:\:\:\,\,}{\,\!}\!\!\!\!\!\!\!\!\!\!\,\overbracket[.5pt]{\:\:}{\,\!}}
\newcommand{\matchD}{\overbracket[.5pt][5pt]{\:\:}{\,\!}\!\!\!\!\:\:\!\overbracket[.5pt][5pt]{\:\:}{\,\!}\!\!\!\!\:\:\!\overbracket[.5pt][5pt]{\:\:}{\,\!}}

\usepackage[capitalise]{cleveref}
\creflabelformat{enumi}{#2#1#3)}

\theoremstyle{plain}
\newtheorem{theorem}{Theorem}
\newtheorem{proposition}[theorem]{Proposition}
\newtheorem{lemma}[theorem]{Lemma}

\theoremstyle{definition}

\theoremstyle{remark}
\newtheorem{remark}{Remark}

\title[]{Limiting spectral distribution of a new random matrix model with dependence across rows and columns}
\author{Oliver Pfaffel}
\author{Eckhard Schlemm}
\address{TUM Institute for Advanced Study \& Zentrum Mathematik, Technische Universit\"at M\"unchen, Boltzmannstra\ss e 3, 85748 Garching bei M\"unchen, Germany.}
\thanks{Both authors gratefully acknowledge financial support from Technische Universit\"at M\"unchen - Institute for Advanced Study funded by the German Excellence Initiative and from the International Graduate School of Science and Engineering}

\begin{document}
\begin{abstract}
We introduce a random matrix model where the entries are dependent across both rows and columns. More precisely, we investigate matrices of the form $\X=(X_{(i-1)n+t})_{it}\in\R^{p\times n}$ derived from a linear process $X_t=\sum_j c_j Z_{t-j}$, where the $\{Z_t\}$ are independent random variables with bounded fourth moments. We show that, when both $p$ and $n$ tend to infinity such that the ratio $p/n$ converges to a finite positive limit $y$, the empirical spectral distribution of $p^{-1}\X\X^{\T}$ converges almost surely to a deterministic measure. This limiting measure, which depends on $y$ and the spectral density of the linear process $X_t$, is characterized by an integral equation for its Stieltjes transform. The matrix $p^{-1}\X\X^{\T}$ can be interpreted as an approximation to the sample covariance matrix of a high\hyp{}dimensional process whose components are independent copies of $X_t$.
\end{abstract}

\maketitle

\section{Introduction}
Random matrix theory studies the properties of large random matrices $A=(A_{i,j})_{ij}\in \mathbb{K}^{p\times n}$, for some field $\mathbb{K}$. In this article, the entries $A_{ij}$ are real random variables unless otherwise specified. Commonly, the focus is on asymptotic properties of such matrices as their dimensions tend to infinity. One particularly interesting object of study is the asymptotic distribution of their singular values. Since the squared singular values of $A$ are the eigenvalues of $AA^{\T}$, this is often done by investigating the eigenvalues of $AA^{\T}$, which is called a sample covariance matrix. The spectral characteristics of a $p\times p$ matrix $S$ are conveniently studied via its empirical spectral distribution, which is defined as $F^{S} = p^{-1}\sum_{i=1}^p{\delta_{\lambda_i}}$; here, $\{\lambda_1,\ldots,\lambda_p\}$ are the eigenvalues of $S$, and $\delta_x$ denotes the Dirac measure located at $x$. For some set $B\subset\R$, the figure $F^{S}(B)$ is the number of eigenvalues of $S$ that lie in $B$. The measure $F^{S}$ is considered a random element of the space of probability distributions equipped with the weak topology, and we are interested in its limit as both $n$ and $p$ tend to infinity such that the ratio $p/n$ converges to a finite positive limit $y$.

The first result of this kind can be found in the remarkable paper of Marchenko and Pastur \citep{Marchenko1967}. They showed that $F^{p^{-1}AA^{\T}}$ converges to a non\hyp{}random limiting spectral distribution $\hat F^{p^{-1}AA^{\T}}$ if all $A_{ij}$ are independent, identically distributed, centred random variables with finite fourth moment. Interestingly, the Lebesgue density of $\hat F ^{p^{-1}AA^{\T}}$ is given by an explicit formula which only involves the ratio $y$ and the common variance of $A_{ij}$ and is therefore universal with respect to the distribution of the entries of $A$. Subsequently \citep{wachter1978,yin1986}, the same result was obtained under the weaker moment condition that the entries $A_{ij}$ have finite variance. The requirement that the entries of $A$ be  identically distributed has later been relaxed to a Lindeberg\hyp{}type condition, cf.\ \cref{Z2}. For more details and a comprehensive treatment of random matrix theory we refer the reader to the text books \citet{Anderson2009,bai2010,Mehta2004}.

Recent research has focused on the question to what extent the assumption of independence of the entries of $A$ can be relaxed without compromising the validity of the Marchenko--Pastur law. In \citet{aubrun2006} it was shown that for random matrices $A$ whose rows are independent $\R^n$-valued random variables uniformly distributed on the unit ball of $l_q(\R^n)$, $q>1$, the empirical spectral distribution $F^{p^{-1}AA^{\T}}$ still converges to the same law as in the i.\,i.\,d.\ case. The Marchenko--Pastur law is, however, not stable with respect to more substantial deviations from the independence assumptions.

A very useful tool to characterize the limiting spectral distribution in random matrix models with dependent entries is the Stieltjes transform which, for some measure $\mu$, is defined as the map $s_\mu:\C^+\to\C^+$, $s_\mu(z) = \int_\R(t-z)^{-1}\mu(\dd t)$. A particular, very successful random matrix model exhibiting dependence within the rows was investigated already by \citet{Marchenko1967} and later in greater generality by \citet{pan2010,silverstein1995}: they modelled dependent data as a linear transformation of independent random variables which led to the study of the eigenvalues of random matrices of the form $AHA^{\T}$, where the entries of $A$ are independent, and $H$ is a positive semidefinite population covariance matrix whose spectral distribution converges to a non\hyp{}random limit $\hat F^H$. They found that the Stieltjes transform of the limiting spectral distribution of $p^{-1}AHA^{\T}$ can be characterized as the solution to an integral equation involving only $\hat F^H$ and the ratio $y=\lim p/n$. Another approach,  suggested in \citet{bai2008} and further pursued in \citet{pfaffel2010}, is to model the rows of $A$ independently as stationary linear processes with independent innovations. This structure is interesting because the class of linear processes includes many practically relevant time series models, such as (fractionally integrated) ARMA processes, as special cases. The main result of \citet{pfaffel2010} shows that for this model the limiting spectral distribution depends only on $y$ and the second\hyp{}order properties of the underlying linear process.

All results for independent rows with dependent row entries also hold with minor modifications for the case where $A$ has independent columns with dependent column entries. This is due to the fact that the matrices $AA^{\T}$ and $A^{\T} A$ have the same non\hyp{}zero eigenvalues. 

In contrast, there are only very few results dealing with random matrix models where the entries are dependent across both rows and columns. The case where $A$ is given as the result of a two\hyp{}dimensional linear filter applied to an array of independent complex Gaussian random variables is considered in \citet{hachem2005}. They use the fact that $A$ can be transformed to a random matrix with uncorrelated, non\hyp{}identically distributed entries. Because of the assumption of Gaussianity the entries are in fact independent, and so an earlier result by the same authors \citep{hachem2006} can be used to obtain the asymptotic distribution of the eigenvalues of $p^{-1}AA^*$. In the context of operator\hyp{}valued free probability theory, \citet{speicher2008} succeeded in characterizing the limiting spectral distribution of block Wishart matrices through a quadratic matrix equation for the corresponding operator-valued Stieltjes transform.

A parallel line of research focuses on the spectral statistics of large symmetric or Hermitian square matrices with dependent entries, thus extending Wigner's \citep{wigner1958} seminal result for the i.\,i.\,d.\ case. Models studied in this context include random Toeplitz, Hankel and circulant matrices \citep[][and references therein]{bose2009,Bryc2006,meckes2007} as well as approaches allowing for a more general dependence structure \citep{anderson2008, hofmann2008}.

In \citet{pfaffel2010}, the authors considered sample covariance matrices of high\hyp{}dimensional stochastic processes, the components of which are modelled by independent infinite\hyp{}order moving average processes with identical second\hyp{}order characteristics. In practice, it is often not possible to observe all components of such a high\hyp{}dimensional process, and the sample covariance matrix can then not be computed. To solve this problem when only one component is observed, it seems reasonable to partition one long observation record of that observed component of length $pn$ into $p$ segments of length $n$, and to treat the different segments as if they were records of the unobserved components. We show that this approach is valid and leads to the correct asymptotic eigenvalue distribution of the sample covariance matrix if the components of the underlying process are modelled as independent moving averages. 

We are thus led to investigate a model of random matrices $\X$ whose entries are dependent across both rows and columns, and which is not covered by the results mentioned above. The entries of the random matrix under consideration are defined in terms of a single linear stochastic process, see \cref{section-result} for a precise definition. Without assuming Gaussianity we prove almost sure convergence of the empirical spectral distribution of $p^{-1}\X\X^{\T}$ to a deterministic limiting measure and characterize the latter via an integral equation for its Stieltjes transform, which only depends on the asymptotic aspect ratio of the matrix and the second\hyp{}order properties of the underlying linear process. Our result extends the class of random matrix models for which the limiting spectral distribution can be identified explicitly by a new, theoretically appealing model. It thus contributes to laying the ground for further research into more general random matrix models with dependent, non\hyp{}identically distributed entries.

\paragraph*{\bf Outline}
In \cref{section-result} we give a precise definition of the random matrix model we investigate and state the main result about its limiting spectral distribution. The proof of the main theorem as well as some auxiliary results are presented in \cref{section-proofs}. Finally, in \cref{section-heuristic}, we indicate how our result could be obtained in an alternative way from a similar random matrix model with independent rows.

\paragraph*{\bf Notation}
We use $\E$ and $\var$ to denote expected value and variance. Where convenient, we also write $\mu_{1,X}$ and $\mu_{2,X}$ for the first and second moment, respectively, of a random variable $X$. The symbol $\1_m$, $m$ a natural number, stands for the $m\times m$ identity matrix. For the trace of a matrix $S$ we write $\tr S$. For sequences of matrices $(S_n)_{n}$ we will suppress the dependence on $n$ where this does not cause ambiguity; the sequence of associated spectral distributions is denoted by $F^{S}$, and for their weak limit, provided it exists, we write $\hat F^S$. It will also be convenient to use asymptotic notation: for two sequences of real numbers $(a_n)_n$, $(b_n)_n$ we write $a_n=O(b_n)$ to indicate that there exists a constant $C$ which is independent of $n$, such that $a_n\leq C b_n$ for all $n$. We denote by $\Z$ the set of integers and by $\N$, $\R$, and $\C$ the sets of natural, real, and complex numbers, respectively. $\Im z$ stands for the imaginary part of a complex number $z$, and $\C^+$ is defined as $\{z\in\C:\Im z>0\}$. The indicator of an expression $\mathcal{E}$ is denoted by $I_{\{\mathcal{E}\}}$ and defined to be one if $\mathcal{E}$ is true and zero otherwise.

\section{A new random matrix model}
\label{section-result}
For a sequence $(Z_t)_{t\in\Z}$ of independent real random variables and real coefficients $(c_j)_{j\in\N\cup\{0\}}$, the linear process $(X_t)_{t\in \Z}$ and the $p\times n$ matrix $\X$ are defined by $X_t=\sum_{j=0}^\infty c_jZ_{t-j}$ and
\begin{equation}
\label{eq-defX}
\X = (\X_{i,t})_{it} =(X_{(i-1)n+t})_{it}= \left(\begin{array}{ccc}X_1&\ldots&X_n\\X_{n+1}&\ldots&X_{2n}\\\vdots& & \vdots\\X_{(p-1)n+1}&\ldots&X_{pn}\end{array}\right)\in\R^{p\times n}.
\end{equation}
The interesting feature about this matrix $\X$ is that its entries are dependent across both rows and columns. In contrast to models considered in \citep{bai2008,hachem2006,pfaffel2010}, not all entries far away from each other are asymptotically independent, e.\,g., the correlation between the entries $\X_{i,n}$ and $\X_{i+1,1}$, $i=1,\ldots,p-1$, does not depend on $n$. We will investigate the asymptotic distribution of the eigenvalues of $p^{-1}\X\X^{\T}$ as both $p$ and $n$ tend to infinity such that their ratio $p/n$ converges to a finite, positive limit $y$. We assume that the sequence $(Z_t)_t$ satisfies
\begin{align}
\E Z_t=0,\quad \E Z_t^2=1, \quad\text{ and }\quad \sigma_4\coloneqq\sup_t\E Z_t^4<\infty,
\label{Z1}
\end{align}
and that the following Lindeberg\hyp{}type condition is satisfied: for each $\epsilon>0$, 
\begin{align}
\frac{1}{pn} \sum_{t=1}^{pn}\E\left(Z_t^2 I_{\{Z_t^2\geq \epsilon n\}}\right) \to 0,\quad \text{ as }\quad n\to\infty.
\label{Z2}
\end{align}
Condition \labelcref{Z2} is satisfied if all $\{Z_t\}$ are identically distributed, but that is not necessary. As it turns out, the limiting spectral distribution of $p^{-1}\X\X^{\T}$ depends only on $y$ and the second\hyp{}order structure of the underlying linear process $X_t$, which we now recall: its auto\hyp{}covariance function $\gamma:\Z\to\R$ is defined by $\gamma(h) = \E X_0X_h = \sum_{j=0}^\infty{c_jc_{j+|h|}}$; its spectral density $f:[0,2\pi]\to\R$ is the Fourier transform of $\gamma$, namely $f(\omega) = \sum_{h\in\Z}{\gamma(h)\ee^{-\ii h\omega}}$. The following is the main result of the paper.
\begin{theorem}
\label{thm-maintheorem}
Let $X_{t}=\sum_{j=0}^\infty{c_j Z_{t-j}}$, $t\in\Z$, be a linear stochastic process with continuously differentiable spectral density $f$, and let the matrix $\X\in\R^{p\times n}$ be given by \cref{eq-defX}. Assume that
\begin{enumerate}[i)]
  \item\label{maintheorem-item1} the sequence $(Z_t)_t$ satisfies conditions \labelcref{Z1,Z2},
  \item\label{maintheorem-item2} there exist positive constants $C,\delta$ such that $|c_j|\leq C(j+1)^{-1-\delta}$, for all $j\in\N\cup\{0\}$.
\end{enumerate}
Then, as $n$ and $p$ tend to infinity such that the ratio $p/n$ converges to a finite positive limit $y$, the empirical spectral distribution of $p^{-1}\X\X^{\T}$ converges almost surely to a non\hyp{}random probability distribution $\hat{F}$ with bounded support. The Stieltjes transform $z\mapsto s_{\hat F}(z)$ of $\hat F$ is the unique mapping $\C^+\to\C^+$ satisfying
\begin{equation}
\label{eq-stieltjes}
\frac{1}{s_{\hat F}(z)} = - z +  y \int_0^{2\pi} \frac{f(\omega)}{1+f(\omega) s_{\hat F}(z)}\dd\omega.
\end{equation}
\end{theorem}

\begin{remark}
The assumption that the coefficients $(c_j)_j$ decay at least polynomially is not very restrictive; it allows, e.\,g., for $X_t$ to be an ARMA or fractionally integrated ARMA process, which exhibits long\hyp{}range dependence \citep{granger1980,Hosking1981}. In the latter case the entries of the matrix $\X$ are long\hyp{}range dependent as well.
\end{remark}
\begin{remark}
It is possible to generalize the proof of \cref{thm-maintheorem} so that the result also holds for non\hyp{}causal processes, where $X_t=\sum_{j=-\infty}^\infty{c_j Z_{t-j}}$. The required changes are merely notational, the only difference in the result is that the auto\hyp{}covariance function is then given by $\sum_{j=-\infty}^\infty c_j c_{j+|h|}$.
\end{remark}
The distribution $\hat{F}$ can be obtained from $s_{\hat F}$ via the Perron--Frobenius inversion formula \citep[Theorem B.8]{bai2010}, which states that for all continuity point $0<a<b$ of $\hat F$, it holds that $\hat{F}([a,b])=\lim_{\epsilon\to 0^+} \int_a^b \Im s_{\hat{F}}(x+\epsilon \ii) \dd x$. In general, the analytic determination of this distribution is not feasible. It is, however, easy to check that for the special case of independent entries one recovers the classical Marchenko--Pastur law.

\section{Proof of Theorem \ref{thm-maintheorem}}
\label{section-proofs}
The strategy in the proof of \cref{thm-maintheorem} is to show that the limiting spectral distribution of $p^{-1}\X\X^{\T}$ is stable under modifications of $\X$ which reduce the sample covariance matrix to the form $p^{-1}\ZZ H\ZZ^{\T}$, for a matrix $\ZZ$ with i.\,i.\,d.\ entries, and some positive definite $H$. To this end we will repeatedly use the following lemma which presents sufficient conditions for the limiting spectral distributions of two sequences of matrices to be equal.
\begin{lemma}[Trace criterion]
\label{lemma-trace}
Let $A_{1,n}$, $A_{2,n}$ be sequences of $p\times n$ matrices, where $p=p_n$ depends on $n$ such that $p_n\to\infty$ as $n\to\infty$. Assume that the spectral distribution $F^{p^{-1}A_{1,n}A_{1,n}^{\T}}$ converges almost surely to a deterministic limit $\hat F^{p^{-1}A_{1,n}A_{1,n}^{\T}}$ as $n$ tends to infinity. If there exists a positive number $\epsilon$ such that
\begin{enumerate}[i)]
 \item\label{lemma-trace-item1} $p^{-4}\E\left[\tr\left(A_{1,n}-A_{2,n}\right)\left(A_{1,n}-A_{2,n}\right)^{\T}\right]^2=O(n^{-1-\epsilon})$,
 \item\label{lemma-trace-item2} $p^{-2}\E\tr A_{i,n}A_{i,n}^{\T}=O(1)$, $i=1,2$, and
 \item\label{lemma-trace-item3} $p^{-4}\var\tr A_{i,n}A_{i,n}^{\T}=O(n^{-1-\epsilon})$, $i=1,2$,
\end{enumerate}
then the spectral distribution of $p^{-1}A_{2,n}A_{2,n}^{\T}$ is convergent almost surely with the same limit $\hat F^{p^{-1}A_{1,n}A_{1,n}^{\T}}$.
\end{lemma}
\begin{proof}
The claim is a direct consequence of Chebyshev's inequality, the first Borel--Cantelli lemma, and \citet[Corollary A.42]{bai2010}
\end{proof}
With the constants $C$ and $\delta$ from assumption \labelcref{maintheorem-item2} of \cref{thm-maintheorem} we define $\oc_j\coloneqq C(j+1)^{-1-\delta}$, such that $|c_j|\leq \oc_j$ for all $j$. Without further reference we will repeatedly use the fact that $j\mapsto \oc_j$ is monotone, that $\sum_{j=1}^\infty\oc_j^\alpha$ is finite for every $\alpha\geq 1$, and that $\sum_{j=n}^\infty{\oc_j^\alpha}$ is of order $O(n^{1-\alpha(1+\delta)})$. Since it is difficult to deal with infinite\hyp{}order moving averages processes directly, it is convenient to truncate the entries of the matrix $\X$ by defining $\wt X_t=\sum_{j=0}^n{c_jZ_{t-j}}$ and $\wt\X=(\wt X_{(i-1)n+t})_{it}$; this is different from the usual truncation of the support of the entries of a random matrix.
\begin{proposition}[Truncation]
\label{truncation}
If the empirical spectral distribution of $p^{-1}\wt\X\wt\X^{\T}$ converges to a limit, then the empirical spectral distribution of $p^{-1}\X\X^{\T}$ converges to the same limit.
\end{proposition}
\begin{proof}
The proof proceeds in two steps in which we verify conditions \labelcref{lemma-trace-item1,lemma-trace-item2,lemma-trace-item3} of \cref{lemma-trace}.
\paragraph{\bf Step 1}
The definitions of $\X$ and $\wt\X$ imply that
\begin{align*}
\Delta_{\X,\wt\X}\coloneqq\frac{1}{p^2}\tr\left(\X-\wt\X\right)\left(\X-\wt\X\right)^{\T} =& \frac{1}{p^2}\sum_{i=1}^p\sum_{t=1}^n\left[\X_{it}-\wt\X_{it}\right]^2 = \frac{1}{p^2}\sum_{i,t=1}^{p,n}\sum_{k,k'=n+1}^\infty{Z_{(i-1)n+t-k}Z_{(i-1)n+t-k'}c_kc_{k'}}.
\end{align*}
We shall show that the second moment of $\Delta_{\X,\wt\X}$ is of order at most $n^{-2-2\delta}$. Since
\begin{equation}
\label{eq-fubinitrunc}
\sum_{\substack{k,k'\\m,m'}=n+1}^\infty{\E\left|Z_{(i-1)n+t-k}Z_{(i-1)n+t-k'}Z_{(i'-1)n+t'-m}Z_{(i'-1)n+t'-m'}\right||c_k||c_{k'}||c_m||c_{m'}|}\leq\sigma_4\left[\sum_{k=0}^\infty|c_k|\right]^4<\infty,
\end{equation}
we can apply Fubini's theorem to interchange expectation and summation in the computation of
\begin{equation}
\label{eq-EDeltaX2}
\mu_{2,\Delta_{\X,\wt\X}}\coloneqq\E\Delta_{\X,\wt\X}^2 = \frac{1}{p^4}\sum_{\substack{i,i'\\t,t'}=1}^{p,n}\sum_{\substack{k,k'\\m,m'}=n+1}^\infty{\E\left[Z_{(i-1)n+t-k}Z_{(i-1)n+t-k'}Z_{(i'-1)n+t'-m}Z_{(i'-1)n+t'-m'}\right]c_kc_{k'}c_mc_{m'}}.
\end{equation}
Since the $\{Z_t\}$ are independent, the expectation in that sum is non\hyp{}zero only if all four $Z$ are the same or else one can match the indices in two pairs. In the latter case we distinguish three cases according to which factor the first $Z$ is paired with. This leads to the additive decomposition
\begin{equation}
\label{eq-mu2DeltaXXtilde}
\mu_{2,\Delta_{\X,\wt\X}} = \mu_{2,\Delta_{\X,\wt\X}}^{\matchD}+\mu_{2,\Delta_{\X,\wt\X}}^{\matchA}+\mu_{2,\Delta_{\X,\wt\X}}^{\matchB}+\mu_{2,\Delta_{\X,\wt\X}}^{\matchC},
\end{equation}
where the ideograms indicate which of the four factors are equal. For the contribution from all four $Z$ being equal it holds that $k=k'$, $m=m'$, and $(i-1)n+t-k=(i'-1)n+t'-m$, so that
\begin{equation*}
\mu_{2,\Delta_{\X,\wt\X}}^{\matchD} = \frac{\sigma_4}{p^4}\sum_{i,i'}^p\sum_{t,t'=1}^n\sum_{m=\max\{n+1,n+1-(i-i')n-(t-t')\}}^\infty{c_{(i-i')n+(t-t')+m}^2c_m^2}.
\end{equation*}
If we introduce the new summation variables $\delta_i\coloneqq i-i'$ and $\delta_t\coloneqq t-t'$, we obtain
\begin{align*}
\mu_{2,\Delta_{\X,\wt\X}}^{\matchD} =& \frac{\sigma_4}{p^4}\sum_{\delta_i=1-p}^{p-1}\underbrace{(p-|\delta_i|)}_{\leq p}\sum_{\delta_t=1-n}^{n-1}\underbrace{(n-|\delta_t|)}_{\leq n}\sum_{m=\max\{n+1,n+1-\delta_in-\delta_t\}}^\infty{c_{m+\delta_i n+\delta_t}^2c_m^2}.
\end{align*}
If $\delta_i$ is positive, then $\delta_in+\delta_t$ is positive as well; the fact that $|c_j|$ is bounded by $\oc_j$ and the monotonicity of $j\mapsto \oc_j$ imply that $c_{m+\delta_i n+\delta_t}^2\leq \oc_{(\delta_i-1)n}\oc_{\delta_t+n}$ so that the contribution from $\delta_i\geq 1$ can be estimated as
\begin{equation*}
\mu_{2,\Delta_{\X,\wt\X}}^{\matchD,+} \leq  \underbrace{\frac{\sigma_4 n}{p^3}}_{=O(n^{-2})}\underbrace{\sum_{\delta_i=1}^{p-1}\oc_{(\delta_i-1)n}}_{=O(n^{-1-\delta})}\underbrace{\sum_{\delta_t=1}^{2n-1}\oc_{\delta_t}}_{=O(1)}\underbrace{\sum_{m=n+1}^\infty \oc_m^2}_{=O(n^{-1-2\delta})}=O(n^{-4-3\delta}).
\end{equation*}
An analogous argument shows that the contribution from $\delta_i\leq -1$, denoted by $\mu_{2,\Delta_{\X,\wt\X}}^{\matchD,-}$, is of the same order of magnitude. The contribution to $\mu_{2,\Delta_{\X,\wt\X}}^{\matchD}$ from $\delta_i=0$ is given by
\begin{align*}
\mu_{2,\Delta_{\X,\wt\X}}^{\matchD,\varnothing} =& \frac{\sigma_4n}{p^3}\sum_{\delta_t=1-n}^{n-1}\sum_{m=\max\{n+1,n+1-\delta_t\}}^\infty{c_m^2 c_{m+\delta_t}^2}\leq \underbrace{\frac{\sigma_4n}{p^3}}_{=O(n^{-2})}\left[2\underbrace{\sum_{\delta_t=1}^{n-1}\oc_{\delta}^2}_{=O(1)}\underbrace{\sum_{m=n+1}^\infty \oc_m^2}_{=O(n^{-1-2\delta})}+\underbrace{\sum_{m=n+1}^\infty \oc_m^4}_{=O(n^{-3-4\delta})}\right]=O(n^{-3-2\delta}).
\end{align*}
By combining the last two displays, it follows that $\mu_{2,\Delta_{\X,\wt\X}}^{\matchD}$ is of order $O(n^{-3-2\delta})$. The second term in \cref{eq-mu2DeltaXXtilde} corresponds to $k=k'$, $m=m'$, and $(i-1)n+t-k\neq(i'-1)n+t'-m$. The restriction that not all four factors be equal is taken into account by subtracting $\mu_{2,\Delta_{\X,\wt\X}}^{\matchD}$; consequently,
\begin{equation*}
\mu_{2,\Delta_{\X,\wt\X}}^{\matchA} = \underbrace{\frac{1}{p^4}\sum_{i,i'=1}^p\sum_{t,t'=1}^n}_{=O(1)}\underbrace{\sum_{k,m=n+1}^\infty c_k^2 c_m^2}_{=O(n^{-2-4\delta})} - \mu_{2,\Delta_{\X,\wt\X}}^{\matchD}=O(n^{-2-4\delta}).
\end{equation*}
It remains to analyse $\mu_{2,\Delta_{\X,\wt\X}}^{\matchB}$ which, by symmetry, is equal to $\mu_{2,\Delta_{\X,\wt\X}}^{\matchC}$. If the first factor is paired with the third, the condition for non\hyp{}vanishment becomes $k = m +(i-i')n + t-t'$, $k' = m' +(i-i')n + t-t'$, and $m\neq m'$. Again introducing the new summation variables $\delta_i\coloneqq i-i'$ and $\delta_t\coloneqq t-t'$, we obtain that
\begin{align*}
\mu_{2,\Delta_{\X,\wt\X}}^{\matchB} = \frac{1}{p^4}\sum_{\delta_i=1-p}^{p-1}\underbrace{(p-|\delta_i|)}_{\leq p}\sum_{\delta_t=1-n}^{n-1}\underbrace{(n-|\delta_t|)}_{\leq n}\sum_{m,m'=\max\{n+1,n+1-\delta_in-\delta_t\}}^\infty{c_m c_{m'}c_{m+\delta_i n+\delta_t}c_{m'+\delta_i n+\delta_t}}-\mu_{2,\Delta_{\X,\wt\X}}^{\matchD}.
\end{align*}
As in the analysis of $\mu_{2,\Delta_{\X,\wt\X}}^{\matchD}$ we obtain the contribution from $\delta_i\neq 0$ as
\begin{equation}
\label{eq-mu2DeltaXXtilematchBplusminus}
\left|\mu_{2,\Delta_{\X,\wt\X}}^{\matchB,+}\right|=\left|\mu_{2,\Delta_{\X,\wt\X}}^{\matchB,-}\right| \leq \underbrace{\frac{n}{p^3}}_{=O(n^{-2})} \underbrace{\sum_{\delta_i=1}^{p-1}\oc_{(\delta_i-1)n}}_{=O(n^{-1-\delta})}\underbrace{\sum_{\delta_t=1}^{2n-1}\oc_{\delta_t}}_{=O(1)}\underbrace{\sum_{m,m'=n+1}^\infty \oc_m \oc_{m'}}_{=O(n^{-2\delta})}+\mu_{2,\Delta_{\X,\wt\X}}^{\matchD}=O(n^{-3-2\delta}).
\end{equation}
Finally, for the contribution from $\delta_i=0$ one finds that
\begin{align}
\label{eq-mu2DeltaXXtilematchBnull}
\left|\mu_{2,\Delta_{\X,\wt\X}}^{\matchB,\varnothing}\right| \leq& \frac{n}{p^3}\sum_{\delta_t=1-n}^{n-1}\sum_{m,m'=\max\{n+1,n+1-\delta_t\}}^\infty{|c_mc_{m'}c_{m+\delta_t}c_{m'+\delta_t}|}+\mu_{2,\Delta_{\X,\wt\X}}^{\matchD}\notag\\
\leq & \underbrace{\frac{n}{p^3}}_{=O(n^{-2})}\left[ 2\underbrace{\sum_{\delta_t=1}^{n-1}\oc_{\delta_t}^2}_{=O(1)}\underbrace{\sum_{m,n'=n+1}^\infty{\oc_m\oc_{m'}}}_{=O(n^{-2\delta})} + \underbrace{\sum_{m,n'=n+1}^\infty{\oc_m^2\oc_{m'}^2}}_{=O(n^{-2-4\delta})}\right]+\mu_{2,\Delta}^{\matchD}=O(n^{-2-2\delta}).
\end{align}
The last two displays \labelcref{eq-mu2DeltaXXtilematchBplusminus,eq-mu2DeltaXXtilematchBnull} imply that $\mu_{2,\Delta_{\X,\wt\X}}^{\matchB} = \mu_{2,\Delta_{\X,\wt\X}}^{\matchB,-}+\mu_{2,\Delta_{\X,\wt\X}}^{\matchB,\varnothing}+\mu_{2,\Delta_{\X,\wt\X}}^{\matchB,+}=O(n^{-2-2\delta})$. Thus, $\mu_{2,\Delta_{\X,\wt\X}}$ is of order $O(n^{-2-2\delta})$, as claimed. 
\paragraph{\bf Step 2}
Next we verify assumptions \labelcref{lemma-trace-item2,lemma-trace-item3} of \cref{lemma-trace}, which means that we show that both $\Sigma_{\X} \coloneqq p^{-2}\tr\X\X^{\T}$ and $\Sigma_{\wt\X} \coloneqq p^{-2}\tr\wt\X\wt\X^{\T}$ have bounded first moments and variances of order $n^{-1-\epsilon}$, for some $\epsilon>0$; in fact, $\epsilon$ will turn out to be one. For $\Sigma_{\X}$ we obtain
\begin{align*}
\mu_{1,\Sigma_{\X}}\coloneqq\E\Sigma_{\X} 
 = \frac{1}{p^2}\sum_{i=1}^p\sum_{t=1}^n\sum_{k,k'=0}^\infty{\E\left[Z_{(i-1)n+t-k}Z_{(i-1)n+t-k'}\right]c_kc_{k'}}= \frac{n}{p}\sum_{k=0}^\infty{c_k^2},
\end{align*}
where the change of the order of expectation and summation is valid by Fubini's theorem. Using \cref{eq-fubinitrunc} and Fubini's theorem, the second moment of $\Sigma_{\X}$ becomes
\begin{align*}
\mu_{2,\Sigma_{\X}}\coloneqq\E\Sigma_{\X}^2=\frac{1}{p^4}\sum_{i,i'=1}^p\sum_{t,t'=1}^n\sum_{\substack{k,k'\\m,m'}=0}^\infty{\E\left[Z_{(i-1)n+t-k}Z_{(i-1)n+t-k'}Z_{(i'-1)n+t'-m}Z_{(i'-1)n+t'-m'}\right]c_kc_{k'}c_mc_{m'}}.
\end{align*}
This sum coincides with the expression analysed in \cref{eq-EDeltaX2}, except that here the $k,k',m,m'$ sums start at zero, and not at $n+1$. A straightforward adaptation of the arguments there show that $\mu_{2,\Sigma_{\X}}$ equals $n^2 p^{-2}\left(\sum_{k=0}^\infty{c_k^2}\right)^2+ O(n^{-2})$, and, consequently, that $\var\Sigma_{\X}=\mu_{2,\Sigma_{\X}}-\left(\mu_{1,\Sigma_{\X}}\right)^2=O(n^{-2})$. Analogous computations show that $\E\Sigma_{\wt\X}$ is bounded, and that $\var\Sigma_{\wt\X}=O(n^{-2})$. Thus, conditions \labelcref{lemma-trace-item2,lemma-trace-item3} of \cref{lemma-trace} are verified, and the proof of the proposition is complete.
\end{proof}

Because of \cref{truncation} the problem of determining the limiting spectral distribution of the sample covariance matrix $p^{-1}\X\X^T$ has been reduced to computing the limiting spectral distribution of $p^{-1}\wt\X\wt\X^{\T}$, where now, for fixed $n$, the matrix $\wt\X$ depends on only finitely many of the noise variables $Z_t$. The fact that the entries of $\wt\X$ are finite\hyp{}order moving average processes and therefore linearly dependent on the $Z_t$ allows for $\wt\X$ to be written as a linear transformation of the i.\,i.\,d.\ matrix $\ZZ\coloneqq(Z_{(i-2)n+t})_{i=1,\ldots,p+1,t=1,\ldots,n}$. We emphasize that $\ZZ$, in contrast to $\X$ and $\wt\X$, is a $(p+1)\times n$ matrix; this is necessary because the entries in the first row of $\wt\X$ depend on noise variables with negative indices, up to and including $Z_{1-n}$. In order to formulate the transformation that maps $\ZZ$ to $\wt\X$ concisely in the next lemma, we define the matrices $K_n=\left(\begin{array}{cc}0 & 0 \\ \1_{n-1} & 0\end{array}\right)\in \R^{n\times n}$, as well as the polynomials $\chi_n(z) = c_0+c_1z + \ldots + c_{n}z^{n}$ and $\bar\chi_n(z) = z^n\chi\left(1/z\right) = c_n + c_{n-1}z + \ldots + c_0 z^n$.
\begin{lemma}
\label{lemma-wtXrep}
With $\wt\X$, $\ZZ$, $K_n$ and $\chi_n$, $\bar\chi_n$ defined as before it holds that
\begin{equation}
\label{eq-Xtilderep}
\wt\X = \left[\begin{array}{cccc}0 & \1_p & \1_p & 0\end{array}\right]\left(\begin{array}{cc}\ZZ & 0 \\ 0 & \ZZ\end{array}\right)\left[\begin{array}{c}\chi_n\left(K_n^{\T}\right)\\\bar\chi_n\left(K_n\right)\end{array}\right].
\end{equation}
\end{lemma}
\begin{proof}
Let $s_N:\R^N\to\R^N$ be the right shift operator defined by $s_N(v_1,\ldots,v_N) = (0, v_1,\ldots,v_{N-1})$ and for positive integers $r,s$ denote by $\operatorname{vec}_{r,s}:\R^{r\times s}\to\R^{rs}$ the bijective linear operator that transforms a matrix into a vector by horizontally concatenating its subsequent rows, starting with the first one. The operator $S_{r,s}:\R^{r\times s}\to \R^{r\times s}$ is then defined as $S_{r,s}=\operatorname{vec}_{r,s}^{-1}\circ s_{rs}\circ\operatorname{vec}_{r,s}$. This operator shifts all entries of a matrix to the right except for the entries in the last column, which are shifted down and moved into the first column. For $k=1,2,\ldots$, the operator $S_{r,s}^k$ is defined as the $k$\hyp{}fold composition of $S_{r,s}$. In the following, we write $S\coloneqq S_{p+1,n}$. With this notation it is clear that $\wt\X = \left[\begin{array}{cc} 0 & \1_p \end{array}\right]\chi_n(S)\ZZ$. In order to obtain \cref{eq-Xtilderep}, we observe that the action of $S$ can be written in terms of matrix multiplications as $S\ZZ= K_{p+1}\ZZ E+\ZZ K_n^{\T}$, where the entries of the $n\times n$ matrix $E$ are all zero except for a one in the lower left corner. Using the fact that $E({K_n^{\T}})^m E$ is zero for every non\hyp{}negative integer $m$ it follows by induction that $S^k$, $k=1,\ldots,n$, acts like
\begin{align*}
S^k\ZZ =& \ZZ {\left(K_n^{\T}\right)}^k+K_{p+1}\ZZ\sum_{i=1}^k{{\left(K_n^{\T}\right)}^{k-i}E{\left(K_n^{\T}\right)}^{i-1}}
  = \left[\begin{array}{cc} \1_{p+1} & K_{p+1} \end{array}\right]\left(\begin{array}{cc}\ZZ & 0 \\ 0 & \ZZ\end{array}\right)\left[\begin{array}{c} {\left(K_n^{\T}\right)}^k \\ K_n^{n-k} \end{array}\right].
\end{align*}
This implies that
\begin{align*}
\wt\X =& 
  \left[\begin{array}{cc} 0 & \1_p \end{array}\right]\left[\begin{array}{cc} \1_{p+1} & K_{p+1} \end{array}\right]\left(\begin{array}{cc}\ZZ & 0 \\ 0 & \ZZ\end{array}\right)\sum_{k=0}^n{c_k\left[\begin{array}{c} {\left(K_n^{\T}\right)}^k \\ K_n^{n-k} \end{array}\right]}
  = \left[\begin{array}{cccc}0 & \1_p & \1_p & 0\end{array}\right]\left(\begin{array}{cc}\ZZ & 0 \\ 0 & \ZZ\end{array}\right)\left[\begin{array}{c}\chi_n\left(K_n^{\T}\right)\\\bar\chi_n\left(K_n\right)\end{array}\right]
\end{align*}
and completes the proof.
\end{proof}

While the last lemma gives an explicit description of the relation between $\ZZ$ and $\wt\X$, it is impractical for directly determining the limiting spectral distribution of $p^{-1}\wt\X\wt\X^{\T}$. The reason is that $\ZZ$ appears twice in the central block-diagonal matrix and is moreover multiplied by some deterministic matrices from both the left and the right. The LSD of the product of three random matrices has been computed in the literature \citep{Zhang2006}, but this result is not applicable in our situation due to the appearance of the random block matrix in \cref{eq-Xtilderep}.  Sample covariance matrices derived from random block matrices have been considered in \citet{speicher2008}. However, they only treat the Gaussian case and, more importantly, do not cover the case of a non\hyp{}trivial population covariance matrix.
We are thus not aware of any result allowing to derive the LSD of $p^{-1}\wt\X\wt\X^{\T}$ directly from \cref{lemma-wtXrep}.

The next proposition allows us to circumvent this problem. It is shown that, at least asymptotically and at the cost of slightly changing the size of the involved matrices, one can simplify the structure of $\wt\X$ so that $\ZZ$ appears only once and is multiplied by a deterministic matrix only from the right.

\begin{proposition}
\label{proposition-whX}
Let $\ZZ$, $K_n$ and $\chi_n$, $\bar\chi_n$ be as before and define the matrix $\wh\X \coloneqq \ZZ\Omega\in\R^{(p+1)\times(n+1)}$, where
\begin{equation}
\label{eq-DefOmega}
\Omega = \left[\begin{array}{cccc}0 & \1_n & \1_n & 0\end{array}\right] \left[\begin{array}{c}\chi_{n+1}\left(K_{n+1}^{\T}\right)\\\bar\chi_{n+1}\left(K_{n+1}\right)\end{array}\right]\in\R^{n\times (n+1)}.
\end{equation}
If the empirical spectral distribution of $p^{-1}\wh\X\wh\X^{\T}$ converges to a limit, then the empirical spectral distribution of $p^{-1}\wt\X\wt\X^{\T}$ converges to the same limit.
\end{proposition}
\begin{proof}
In order to be able to compare the limiting spectral distributions of $p^{-1}\wt\X\wt\X^{\T}$ and $p^{-1}\hat\X\hat\X^{\T}$ in spite of their dimensions being different, we introduce the matrix $\wtt\X = \left[\begin{array}{cc}0 & 0\\0&\wt\X\end{array}\right]\in\R^{(p+1)\times(n+1)}$. Clearly, $F^{p^{-1}\wtt\X\wtt\X^{\T}}=(p+1)^{-1}\delta_0+p(p+1)^{-1}F^{p^{-1}\wt\X\wt\X^{\T}}$, which implies equality of the limiting spectral distributions provided either of the two, and hence both, exists. It is therefore sufficient to show that the LSD of $p^{-1}\wh\X\wh\X^{\T}$ and $p^{-1}\wtt\X\wtt\X^{\T}$ are identical; this will be done by verifying the three conditions of \cref{lemma-trace}. The remainder of the proof will be divided in two parts. In the first part we check the validity of assumption \labelcref{lemma-trace-item1} about the difference $\wh\X-\wtt\X$, whereas in the second one we consider the terms $\tr\wh\X\wh\X^{\T}$ and $\tr\wtt\X\wtt\X^{\T}$, which appear in conditions \labelcref{lemma-trace-item2,lemma-trace-item3}.
\paragraph{\bf Step 1}
Using the definitions of $\wh{\X}$ and  $\wtt\X$, it follows that
\begin{align}
\label{eq-trxhxttminus}
\Delta_{\wh\X,\wtt\X}\coloneqq&\frac{1}{p^2}\tr\left(\wh\X-\wtt\X\right)\left(\wh\X-\wtt\X\right)^{\T} = \frac{1}{p^2}\sum_{i=1}^{p+1}\sum_{j=1}^{n+1}  {\left[\wh\X_{ij}-\wtt\X_{ij}\right]^2}\notag\\
  \leq& \frac{2}{p^2}\sum_{i=2}^{p+1}\sum_{j=2}^{n+1}\left[\sum_{k,k'=j}^n{Z_{(i-2)n+k}Z_{(i-2)n+k'}c_{j-k+n+1}c_{j-k'+n+1}}+\sum_{k,k'=j-1}^n{Z_{(i-3)n+k}Z_{(i-3)n+k'}c_{j-k+n-1}c_{j-k'+n-1}}\right]\notag\\
  &+\frac{1}{p^2}\sum_{i=1}^{p+1}{\sum_{k,k'=1}^n{Z_{(i-2)n+k}Z_{(i-2)n+k'}c_{n-k+2}c_{n-k'+2}}}+\frac{2}{p^2}\sum_{j=2}^{n+1}{\sum_{k,k'=1}^{j-1}{Z_{-n+k}Z_{-n+k'}c_{j-k-1}c_{j-k'-1}}} \notag\\ &+\frac{2}{p^2}\sum_{j=2}^{n+1}{\sum_{k,k'=j}^n{Z_{-n+k}Z_{-n+k'}c_{j-k+n+1}c_{j-k'+n+1}}} \eqqcolon \sum_{i=1}^5\Delta_{\wh\X,\wtt\X}^{(i)},
\end{align}
where the elementary inequality $(a+b)^2\leq 2a^2+2b^2$ was used twice. In order to show that the variances of expression \labelcref{eq-trxhxttminus} are summable, we consider each term in turn. For the second moment of the first term of \cref{eq-trxhxttminus} we obtain
\begin{align*}
\mu_{2,\Delta_{\wh\X,\wtt\X}^{(1)}}\coloneqq\E\left(\Delta_{\wh\X,\wtt\X}^{(1)}\right)^2= 
 & \frac{4}{p^4}\sum_{i,i'=2}^{p+1}\sum_{j,j'=2}^{n+1}\sum_{k,k'=1}^{n-j+1}\sum_{m,m'=1}^{n-j'+1}\E\left[Z_{(i-1)n-k+1}Z_{(i-1)n-k'+1}Z_{(i'-1)n-m+1}Z_{(i'-1)n-m'+1}\right]\times\\
  &\qquad\qquad\qquad\qquad\qquad\qquad c_{j+k}c_{j+k'}c_{j'+m}c_{j'+m'}.
\end{align*}
As before we consider all configurations where above expectation is not zero. The expectation equals $\sigma_4$ if $i=i'$ and $k,k',m,m'$ are equal, and, hence,
\begin{align*}
\mu_{2,\Delta_{\wh\X,\wtt\X}^{(1)}}^{\matchD} \leq \frac{4\sigma_4}{p^4} \sum_{i=2}^{p+1} \sum_{k=1}^{n} \left( \sum_{j=2}^{n+1} c_{j+k}^2 \right)^2 \leq \frac{4\sigma_4}{p^3}  \sum_{k=1}^{n} \oc_{k}^2  \left( \sum_{j=2}^{n+1} \oc_{j} \right)^2 = O(n^{-3}).
\end{align*}
The expectation is one if the four $Z$ can be collected in two non\hyp{}equal pairs. The first term equals the second, and the third equals the fourth if $k=k'$ and $m=m'$, and thus
\begin{align*}
\mu_{2,\Delta_{\wh\X,\wtt\X}^{(1)}}^{\matchA} &= \frac{4}{p^4}\sum_{i,i'=2}^{p+1}\sum_{j,j'=2}^{n+1}\sum_{k=1}^{n-j+1}\sum_{m=1}^{n-j'+1} c_{j+k}^2 c_{j'+m}^2 - \mu_{2,\Delta_{\wh\X,\wtt\X}^{(1)}}^{\matchD}=\frac{4}{p^2} \left( \sum_{j=2}^{n+1}\sum_{k=1}^{n-j+1} c_{j+k}^2 \right)^2 - \mu_{2,\Delta_{\wh\X,\wtt\X}^{(1)}}^{\matchD} = O(n^{-2}).
\end{align*}
Likewise, the contribution from pairing the first factor with the third, and the second with the fourth, can be estimated as
\begin{align*}
\left|\mu_{2,\Delta_{\wh\X,\wtt\X}^{(1)}}^{\matchB}\right| &\leq \frac{4}{p^4}\sum_{i'=2}^{p+1}\sum_{j,j'=2}^{n+1}\sum_{k,k'=1}^{n} |c_{j+k}c_{j+k'} c_{j'+k}c_{j'+k'}| + \mu_{2,\Delta_{\wh\X,\wtt\X}^{(1)}}^{\matchD} \leq \frac{4}{p^3} \left(\sum_{j=1}^{n+1}\oc_{j} \right)^4 + \mu_{2,\Delta_{\wh\X,\wtt\X}^{(1)}}^{\matchD} = O(n^{-3}).
\end{align*}
Obviously, the configuration $\mu_{2,\Delta_{\wh\X,\wtt\X}^{(1)}}^{\matchC}$ can be handled the same way as $\mu_{2,\Delta_{\wh\X,\wtt\X}^{(1)}}^{\matchB}$ above. Thus we have shown that the second moment of $\Delta_{\wh\X,\wtt\X}^{(1)}$, the first term in \cref{eq-trxhxttminus}, is of order $n^{-2}$. This can be shown for the second term in \cref{eq-trxhxttminus} in the same way. We now consider the second moment of the third term in \cref{eq-trxhxttminus}:
\begin{align*}
\mu_{2,\Delta_{\wh\X,\wtt\X}^{(3)}}\coloneqq&\E\left( \Delta_{\wh\X,\wtt\X}^{(3)} \right)^2 = \frac{1}{p^4}\sum_{i,i'=1}^{p+1}\sum_{\substack{k,k'\\m,m'}=1}^n\E\left[Z_{(i-2)n+k}Z_{(i-2)n+k'}Z_{(i'-2)n+m}Z_{(i'-2)n+m'}\right] c_{n-k+2}c_{n-k'+2}c_{n-m+2}c_{n-m'+2}.
\end{align*}
Distinguishing the same cases as before, we have $\mu_{2,\Delta_{\wh\X,\wtt\X}^{(3)}}^{\matchD} = \sigma_4\frac{p+1}{p^4} \sum_{k=1}^n c_{n-k+2}^4 = O(n^{-3})$ and, thus,
\begin{align*}
\mu_{2,\Delta_{\wh\X,\wtt\X}^{(3)}}^{\matchA} &= \frac{(p+1)^2}{p^4} \left( \sum_{k=1}^n c_{n-k+2}^2 \right)^2 - \mu_{2,\Delta_{\wh\X,\wtt\X}^{(3)}}^{\matchD} = O(n^{-2}),
\end{align*}
as well as $\mu_{2,\Delta_{\wh\X,\wtt\X}^{(3)}}^{\matchC} = \mu_{2,\Delta_{\wh\X,\wtt\X}^{(3)}}^{\matchB} = O(n^{-3})$. Thus, the second moment of the third term in \cref{eq-trxhxttminus} is of order $O(n^{-2})$; repeating the foregoing arguments, it can be seen that the second moments of $\Delta_{\wh\X,\wtt\X}^{(4)}$ and $\Delta_{\wh\X,\wtt\X}^{(5)}$, the two last terms in  \cref{eq-trxhxttminus}, are of order $O(n^{-2})$ as well, so that we have shown that
\begin{equation*}
\frac{1}{p^4}\E\left[\tr\left(\wh\X-\wtt\X\right)\left(\wh\X-\wtt\X\right)^{\T}\right]^2 =\E\left(\Delta_{\wh\X,\wtt\X}\right)^2\leq5\sum_{i=1}^5{\mu_{2,\Delta_{\wh\X,\wtt\X}^{(i)}}}= O(n^{-2}).
\end{equation*}

\paragraph{\bf Step 2}
In this step we shall prove that both $\Sigma_{\wh\X}\coloneqq p^{-2}\tr\wh\X\wh\X^{\T}$ and $\Sigma_{\wtt\X}\coloneqq p^{-2}\tr\wtt\X\wtt\X^{\T}$ have bounded first moments, and that their variances are summable sequences in $n$, i.\,e.\ we check conditions \labelcref{lemma-trace-item2,lemma-trace-item3} of \cref{lemma-trace}. Since $\tr\wtt\X\wtt\X^{\T}$ is equal to $\tr\wt\X\wt\X^{\T}$, the claim about $\Sigma_{\wtt\X}$ has already been shown in the second step of the proof of \cref{truncation}. For the first term one finds, by the definition of $\wh\X$, that
\begin{align*}
\Sigma_{\wh\X} =& \frac{1}{p^2}\sum_{i=1}^{p+1} \sum_{j=1}^{n+1} \left(\sum_{k=1}^{j-1}{Z_{(i-2)n+k}c_{j-k-1}}+\sum_{k=j}^n{Z_{(i-2)n+k}c_{j-k+n+1}}\right)^2 \nonumber\\
\leq& \frac{2}{p^2}\sum_{i=1}^{p+1} \sum_{j=1}^{n+1} \sum_{k,k'=1}^{j-1} Z_{(i-2)n+k}c_{j-k-1} Z_{(i-2)n+k'}c_{j-k'-1} \nonumber\\
&+ \frac{2}{p^2}\sum_{i=1}^{p+1} \sum_{j=1}^{n+1} \sum_{k,k'=j}^{n} Z_{(i-2)n+k}c_{j-k+n+1} Z_{(i-2)n+k'}c_{j-k'+n+1} \eqqcolon \Sigma_{\wh\X}^{(1)}+\Sigma_{\wh\X}^{(2)}.
 \label{trplushut}
\end{align*}
Clearly, the first two moments of $\Sigma_{\wh\X}^{(1)}$ are given by
\begin{equation*}
\mu_{1,\Sigma_{\wh\X}^{(1)}}\coloneqq\E\Sigma_{\wh\X}^{(1)}=\frac{2}{p^2}\sum_{i=1}^{p+1}\sum_{j=1}^{n+1}\sum_{k,k'=1}^{j-1}{\E\left[Z_{(i-2)n+k} Z_{(i-2)n+k'}\right]c_{j-k-1}c_{j-k'-1}}=\frac{2(p+1)}{p^2}\sum_{j=1}^{n+1}\sum_{k=1}^{j-1}{c_{k-1}^2},
\end{equation*}
and
\begin{align*}
\mu_{2,\Sigma_{\wh\X}^{(1)}}\coloneqq&\E\left(\Sigma_{\wh\X}^{(1)}\right)^2=\frac{4}{p^4} \sum_{i,i'=1}^{p+1} \sum_{j,j'=1}^{n+1} \sum_{k,k'=1}^{j-1} \sum_{m,m'=1}^{j'-1} \E(Z_{(i-2)n+k}Z_{(i-2)n+k'}Z_{(i'-2)n+m}Z_{(i'-2)n+m'}) 
\times \\
& \qquad\qquad\qquad\qquad\qquad\qquad\qquad\qquad\qquad c_{j-k-1} c_{j-k'-1} c_{j'-m-1} c_{j'-m'-1}.
\end{align*}
We separately consider the case that all four factors are equal, and the three possible pairings of the four $Z$. If all four $Z$ are equal, it must hold that $i=i'$, $k=k'=m=m'$, with contribution
\begin{align*}
\mu_{2,\Sigma_{\wh\X}^{(1)}}^{\matchD}=&\frac{4\sigma_4}{p^4}\sum_{i=1}^{p+1}\sum_{j,j'=1}^{n+1}\sum_{k=1}^{\min\{j,j'\}-1}c_{j-k-1}^2c_{j'-k-1}^2\\
\leq&\frac{4\sigma_4(p+1)}{p^4}\sum_{j,j'=1}^{n+1}\oc_{j-\min\{j,j'\}}\oc_{j'-\min\{j,j'\}}\sum_{k=1}^{\min\{j,j'\}-1}\oc_{k-1}^2\leq \frac{4\sigma_4(p+1)}{p^4}\sum_{j,j'=1}^{n+1}\oc_{0}\oc_{|j-j'|}\sum_{k=1}^{n}\oc_{k-1}^2.
\end{align*}
Introducing the new summation variable $\delta_j\coloneqq j-j'$, one finds that
\begin{equation}
\label{eq-mu2SigmaXhat1mathD}
\mu_{2,\Sigma_{\wh\X}^{(1)}}^{\matchD} \leq \frac{4\sigma_4(p+1)(n+1)}{p^4}\oc_0\left[\oc_0+2\sum_{\delta_j=1}^n\oc_{\delta_j}\right]\sum_{k=1}^{n}\oc_{k-1}^2 = O(n^{-2}).
\end{equation}
The first factor being paired with the second, and the third with the fourth, means that $k=k'$, $m=m'$, and $m\neq (i-i')n+k$, so that the contribution of this configuration is given by
\begin{align}
\label{eq-mu2SigmaXhat1mathA}
\mu_{2,\Sigma_{\wh\X}^{(1)}}^{\matchA}=&\frac{4}{p^4}\sum_{i,i'=1}^{p+1}\sum_{j,j'=1}^{n+1}\sum_{k=1}^{j-1}\sum_{m=1}^{j'-1}c_{j-k-1}^2c_{j'-m-1}^2-\mu_{2,\Sigma_{\wh\X}^{(1)}}^{\matchD} = \left(\mu_{1,\Sigma_{\wh\X}^{(1)}}\right)^2 +O(n^{-2}).
\end{align}
For the $\matchB$ pairing, the constraints are $i=i'$, $k=m$, $k'=m'$, $k\neq k'$, and the corresponding contribution is
\begin{align}
\label{eq-mu2SigmaXhat1mathB}
\mu_{2,\Sigma_{\wh\X}^{(1)}}^{\matchB}=&\frac{4}{p^4}\sum_{i=1}^{p+1}\sum_{j,j'=1}^{n+1}\sum_{k,k'=1}^{\min\{j,j'\}-1}{c_{j-k-1}c_{j-k'-1}c_{j'-k-1}c_{j'-k'-1}}-\mu_{2,\Sigma_{\wh\X}^{(1)}}^{\matchD}\notag \\
\leq& \frac{4(p+1)}{p^4}\sum_{j,j'=1}^{n+1}\oc_{j-\min\{j,j'\}}\oc_{j'-\min\{j,j'\}}\sum_{k,k'=1}^{\min\{j,j'\}-1}\oc_{k-1}c_{k'-1}+O(n^{-2})\notag\\
\leq& \frac{4(p+1)(n+1)}{p^4}\oc_0\left[\oc_0+2\sum_{\delta_j=1}^n\oc_{\delta_j}\right]\sum_{k,k'=1}^{n}\oc_{k-1}\oc_{k'-1} +O(n^{-2})= O(n^{-2}).
\end{align}
Renaming the summation indices shows that $\mu_{2,\Sigma_{\wh\X}^{(1)}}^{\matchC}=\mu_{2,\Sigma_{\wh\X}^{(1)}}^{\matchB}$. Combining this with the displays \labelcref{eq-mu2SigmaXhat1mathD,eq-mu2SigmaXhat1mathA,eq-mu2SigmaXhat1mathB}, it follows that $\var\Sigma_{\wh\X}^{(1)}=\mu_{2,\Sigma_{\wh\X}^{(1)}}-\mu_{1,\Sigma_{\wh\X}^{(1)}}^2=O(n^{-2})$. Since a very similar reasoning can be applied to $\Sigma_{\wh\X}^{(2)}$, and $p^{-4}\var\tr\wh\X\wh\X^{\T}$ is smaller than $2\var\Sigma_{\wh\X}^{(1)} + 2\var\Sigma_{\wh\X}^{(2)}$, we conclude that $p^{-4}\var\tr\wh\X\wh\X^{\T}$ is of order $O(n^{-2})$.
\end{proof}
The intention behind \cref{proposition-whX} was to allow the application of results about the limiting spectral distribution of matrices of the form $\ZZ H\ZZ^{\T}$, where $\ZZ$ is an i.\,i.\,d.\ matrix, and $H$ is a positive semidefinite matrix. Expressions for the Stieltjes transform of the LSD of such matrices in terms of the LSD of $H$ have been obtained by \citet{Marchenko1967,silverstein1995}, and, in the most general form, by \citet{pan2010}. The next lemma shows that in the current context the population covariance matrix $H$ has the same LSD as the auto\hyp{}covariance matrix $\Gamma$ of the process $X_t$, which is defined in terms of the auto\hyp{}covariance function $\gamma(h) = \sum_{j=0}^\infty{c_jc_{j+|h|}}$ by $\Gamma = (\gamma(i-j))_{ij}$; this correspondence is used to characterize the LSD of $H$ by the spectral density $f$ associated with the coefficients $(c_j)_j$.

\begin{lemma}
\label{lemma-OmegaGamma}
Let $\Omega$ be given by \cref{eq-DefOmega}. The limiting spectral distribution of the matrix $\Omega\Omega^{\T}$ exists and is the same as the limiting spectral distribution of the auto\hyp{}covariance matrix $\Gamma$. It therefore satisfies
\begin{equation}
\label{eq-grayLSD}
\int h(\lambda)\hat F^{\Omega\Omega^{\T}}(\dd\lambda) = \frac{1}{2\pi}\int_0^{2\pi}h(f(\omega))\dd\omega,
\end{equation}
for every continuous function $h$.
\end{lemma}
\begin{proof}
The first claim follows by standard computations from the fact that $\Omega$ is, except for one missing row, a circulant matrix with entries $\Omega_{ij}=c_{n+j-i\mod (n+1)}$, and \citet[Corollaries A.41 and A.42]{bai2010}. The second claim is an application of Szeg\H{o}'s limit theorem about the LSD of Toeplitz matrices; see \citet[Theorem XVIII]{szego1920} for the original result or, e.\,g., \citet[Sections 5.4 and 5.5]{boettcher1999intro} for a modern treatment.
\end{proof}
\begin{proof}[Proof of \cref{thm-maintheorem}]
According to \cref{proposition-whX}, the matrix $\wh\X\wh\X^{\T}$ is of the form $\ZZ\Omega\Omega^{\T}\ZZ^{\T}$, where $\Omega$ is given by \cref{eq-DefOmega}. Using \citet[Theorem 1]{pan2010} and the fact that, by \cref{lemma-OmegaGamma}, the limiting spectral distribution of ${\Omega\Omega^{\T}}$ exists, it follows that the limiting spectral distribution $\hat F^{p^{-1}\wh\X\wh\X^{\T}}$ exists. Therefore, the  combination of \cref{truncation,proposition-whX} shows that the limiting spectral distribution of $p^{-1}\X\X^{\T}$ also exists and is the same as that of $p^{-1}\wh\X\wh\X^{\T}$. \citet[equation (1.2)]{pan2010} thus implies that the Stieltjes transform of $\hat F^{p^{-1}\X\X^{\T}}$ is the unique mapping $s_{\hat F^{p^{-1}\X\X^{\T}}}:\C^+\to\C^+$ which solves
\begin{equation*}
\frac{1}{s_{\hat F^{p^{-1}\X\X^{\T}}}(z)} = - z +  y \int_\R \frac{\lambda}{1+\lambda s_{\hat F^{p^{-1}\X\X^{\T}}}(z)}\hat F^{\Omega\Omega^{\T}}( \dd\lambda),
\end{equation*}
and \cref{eq-grayLSD} from \cref{lemma-OmegaGamma} completes the proof.
\end{proof}

\section{Sketch of an alternative proof of Theorem \ref{thm-maintheorem}}
\label{section-heuristic}
In this section we indicate how \cref{thm-maintheorem} could be proved alternatively using the methods employed in \citet{pfaffel2010}. We denote by $\wt{\X}_{(\alpha)}$ the matrix which is defined as in \cref{eq-defX} but with the linear process being truncated at $\lfloor n^\alpha \rfloor$ with $0<\alpha<1$, i.\,e.\, $\wt{\X}_{(\alpha)}=\left(\sum_{j=0}^{\lfloor n^\alpha \rfloor}{c_j Z_{(i-1)n+t-j}}\right)_{it}$. If $1-\alpha$ is sufficiently small, then an adaptation of the proof of \cref{truncation} to this setting shows that $p^{-1}\X\X^\T$ and $p^{-1}\wt{\X}_{(\alpha)}\wt{\X}_{(\alpha)}^\T$ have the same limiting spectral distribution almost surely. The next step is to partition $\wt{\X}_{(\alpha)}$ into two blocks of dimensions $p\times\lfloor n^\alpha \rfloor$ and $p\times (n-\lfloor n^\alpha \rfloor)$, respectively. If we denote these two blocks by $\wt{\X}_{(\alpha)}^1$ and $\wt{\X}_{(\alpha)}^2$, i.\,e.\, $\wt{\X}_{(\alpha)}=\left[\wt{\X}_{(\alpha)}^1 \,\wt{\X}_{(\alpha)}^2\right]$, then clearly $\wt{\X}_{(\alpha)}\wt{\X}_{(\alpha)}^\T=\wt{\X}_{(\alpha)}^1\left(\wt{\X}_{(\alpha)}^1\right)^\T+\wt{\X}_{(\alpha)}^2\left(\wt{\X}_{(\alpha)}^2\right)^\T$, and an application of \citet[Theorem A.43]{bai2010} yields that
\begin{align*}
\sup_{\lambda\in\R_{\geq 0}}\left|F^{p^{-1}\X\X^\T}([0,\lambda])-F^{p^{-1}\wt{\X}_{(\alpha)}^2\left(\wt{\X}_{(\alpha)}^2\right)^\T}([0,\lambda])\right|\leq\frac{1}{p}\mathrm{rank}\left(\wt{\X}_{(\alpha)}^1\left(\wt{\X}_{(\alpha)}^1\right)^\T\right)\leq \frac{1}{p}\min\left(\lfloor n^\alpha \rfloor,p\right)=O\left(p^{-1}n^\alpha\right)\to 0.
\end{align*}
It therefore suffices to derive the limiting spectral distribution of $p^{-1}\wt{\X}_{(\alpha)}^2\left(\wt{\X}_{(\alpha)}^2\right)^\T$. Since the matrix $\wt{\X}_{(\alpha)}^2$ has independent rows, this could be done by a careful adaptation of the arguments given in \citet{pfaffel2010}. We chose, however, to provide a self\hyp{}contained proof, which also provides intermediate results of independent interest like \cref{proposition-whX}, and we therefore omit the lengthy details of this alternative proof.


\begin{thebibliography}{xx}

\harvarditem[Anderson et~al.]{Anderson, Guionnet \harvardand\
  Zeitouni}{2010}{Anderson2009}
Anderson, G.~W., Guionnet, A. \harvardand\ Zeitouni, O.  \harvardyearleft
  2010\harvardyearright .
\newblock {\em An introduction to random matrices}, Vol. 118 of {\em Cambridge
  Studies in Advanced Mathematics}, Cambridge University Press, Cambridge.

\harvarditem{Anderson \harvardand\ Zeitouni}{2008}{anderson2008}
Anderson, G.~W. \harvardand\ Zeitouni, O.  \harvardyearleft
  2008\harvardyearright .
\newblock A law of large numbers for finite\hyp{}range dependent random
  matrices, {\em Comm. Pure Appl. Math.} {\bf 61}(8):~1118--1154.

\harvarditem{Aubrun}{2006}{aubrun2006}
Aubrun, G.  \harvardyearleft 2006\harvardyearright .
\newblock Random points in the unit ball of {$l\sp n\sb p$}, {\em Positivity}
  {\bf 10}(4):~755--759.

\harvarditem{Bai \harvardand\ Zhou}{2008}{bai2008}
Bai, Z.~D. \harvardand\ Zhou, W.  \harvardyearleft 2008\harvardyearright .
\newblock Large sample covariance matrices without independence structures in
  columns, {\em Stat. Sinica} {\bf 18}(2):~425--442.

\harvarditem{Bai \harvardand\ Silverstein}{2010}{bai2010}
Bai, Z. \harvardand\ Silverstein, J.~W.  \harvardyearleft 2010\harvardyearright
  .
\newblock {\em Spectral analysis of large dimensional random matrices},
  Springer Series in Statistics, second edn, Springer, New York.

\harvarditem[Bose et~al.]{Bose, Subhra~Hazra \harvardand\ Saha}{2009}{bose2009}
Bose, A., Subhra~Hazra, R. \harvardand\ Saha, K.  \harvardyearleft
  2009\harvardyearright .
\newblock Limiting spectral distribution of circulant type matrices with
  dependent inputs, {\em Electron. J. Probab.} {\bf 14}(86):~2463--2491.

\harvarditem{B{\"o}ttcher \harvardand\ Silbermann}{1999}{boettcher1999intro}
B{\"o}ttcher, A. \harvardand\ Silbermann, B.  \harvardyearleft
  1999\harvardyearright .
\newblock {\em Introduction to large truncated {T}oeplitz matrices},
  Universitext, Springer-Verlag, New York.

\harvarditem[Bryc et~al.]{Bryc, Dembo \harvardand\ Jiang}{2006}{Bryc2006}
Bryc, W., Dembo, A. \harvardand\ Jiang, T.  \harvardyearleft
  2006\harvardyearright .
\newblock Spectral measure of large random {H}ankel, {M}arkov and {T}oeplitz
  matrices, {\em Ann. Probab.} {\bf 34}(1):~1--38.

\harvarditem{Granger \harvardand\ Joyeux}{1980}{granger1980}
Granger, C. W.~J. \harvardand\ Joyeux, R.  \harvardyearleft
  1980\harvardyearright .
\newblock An introduction to long\hyp{}memory time series models and fractional
  differencing, {\em J. Time Ser. Anal.} {\bf 1}(1):~15--29.

\harvarditem[Hachem et~al.]{Hachem, Loubaton \harvardand\
  Najim}{2005}{hachem2005}
Hachem, W., Loubaton, P. \harvardand\ Najim, J.  \harvardyearleft
  2005\harvardyearright .
\newblock The empirical eigenvalue distribution of a {G}ram matrix: from
  independence to stationarity, {\em Markov Process. Related Fields} {\bf
  11}(4):~629--648.

\harvarditem[Hachem et~al.]{Hachem, Loubaton \harvardand\
  Najim}{2006}{hachem2006}
Hachem, W., Loubaton, P. \harvardand\ Najim, J.  \harvardyearleft
  2006\harvardyearright .
\newblock The empirical distribution of the eigenvalues of a {G}ram matrix with
  a given variance profile, {\em Ann. Inst. H. Poincar\'e Probab. Statist.}
  {\bf 42}(6):~649--670.

\harvarditem{Hofmann-Credner \harvardand\ Stolz}{2008}{hofmann2008}
Hofmann-Credner, K. \harvardand\ Stolz, M.  \harvardyearleft
  2008\harvardyearright .
\newblock Wigner theorems for random matrices with dependent entries: ensembles
  associated to symmetric spaces and sample covariance matrices, {\em Electron.
  Commun. Probab.} {\bf 13}:~401--414.

\harvarditem{Hosking}{1981}{Hosking1981}
Hosking, J. R.~M.  \harvardyearleft 1981\harvardyearright .
\newblock Fractional differencing, {\em Biometrika} {\bf 68}(1):~165--176.

\harvarditem{Marchenko \harvardand\ Pastur}{1967}{Marchenko1967}
Marchenko, V.~A. \harvardand\ Pastur, L.~A.  \harvardyearleft
  1967\harvardyearright .
\newblock Distribution of eigenvalues in certain sets of random matrices, {\em
  Mat. Sb. (N.S.)} {\bf 72(114)}(4):~507--536.

\harvarditem{Meckes}{2007}{meckes2007}
Meckes, M.~W.  \harvardyearleft 2007\harvardyearright .
\newblock On the spectral norm of a random {T}oeplitz matrix, {\em Electron.
  Comm. Probab.} {\bf 12}:~315--325.

\harvarditem{Mehta}{2004}{Mehta2004}
Mehta, M.~L.  \harvardyearleft 2004\harvardyearright .
\newblock {\em Random matrices}, Pure and Applied Mathematics, third edn,
  Elsevier/Academic Press, Amsterdam.

\harvarditem{Pan}{2010}{pan2010}
Pan, G.  \harvardyearleft 2010\harvardyearright .
\newblock Strong convergence of the empirical distribution of eigenvalues of
  sample covariance matrices with a perturbation matrix, {\em J. Multivar.
  Anal.} {\bf 101}(6):~1330--1338.

\bibitem[Pfaffel and Schlemm(2011)]{pfaffel2010}
Pfaffel, O. and Schlemm, E.
\newblock Eigenvalue distribution of large sample covariance matrices of linear
  processes.
\newblock \emph{Probab. Math. Statist.}, 31 (2):\penalty0 313--329, 2011.

\harvarditem[Rashidi~Far et~al.]{Rashidi~Far, Oraby, Bryc \harvardand\
  Speicher}{2008}{speicher2008}
Rashidi~Far, R., Oraby, T., Bryc, W. \harvardand\ Speicher, R.
  \harvardyearleft 2008\harvardyearright .
\newblock On slow\hyp{}fading {MIMO} systems with nonseparable correlation,
  {\em IEEE Trans. Inform. Theory} {\bf 54}(2):~544--553.

\harvarditem{Silverstein \harvardand\ Bai}{1995}{silverstein1995}
Silverstein, J.~W. \harvardand\ Bai, Z.~D.  \harvardyearleft
  1995\harvardyearright .
\newblock On the empirical distribution of eigenvalues of a class of
  large\hyp{}dimensional random matrices, {\em J. Multivar. Anal.} {\bf
  54}(2):~175--192.

\harvarditem{Szeg\H{o}}{1920}{szego1920}
Szeg\H{o}, G.  \harvardyearleft 1920\harvardyearright .
\newblock Beitr\"age zur {T}heorie der {T}oeplitzschen {F}ormen, {\em Math. Z.}
  {\bf 6}(3-4):~167--202.

\harvarditem{Wachter}{1978}{wachter1978}
Wachter, K.~W.  \harvardyearleft 1978\harvardyearright .
\newblock The strong limits of random matrix spectra for sample matrices of
  independent elements, {\em Ann. Probab.} {\bf 6}(1):~1--18.

\harvarditem{Wigner}{1958}{wigner1958}
Wigner, E.~P.  \harvardyearleft 1958\harvardyearright .
\newblock On the distribution of the roots of certain symmetric matrices, {\em
  Ann. of Math. (2)} {\bf 67}(2):~325--327.

\harvarditem{Yin}{1986}{yin1986}
Yin, Y.~Q.  \harvardyearleft 1986\harvardyearright .
\newblock Limiting spectral distribution for a class of random matrices, {\em
  J. Multivar. Anal.} {\bf 20}(1):~50--68.

\harvarditem{Zhang}{2006}{Zhang2006}
Zhang, L.  \harvardyearleft 2006\harvardyearright .
\newblock {\em Spectral Analysis of large dimensional random matrices}, PhD
  thesis, National University of Singapore.

\end{thebibliography}

\end{document}